\documentclass[11pt]{amsart}

\usepackage{mathrsfs}
\usepackage{palatino}
\usepackage{url}
\usepackage[all,cmtip,2cell]{xy}
\UseTwocells
\usepackage{graphicx}
\usepackage{amssymb}

\newtheorem{global-theorem}{Theorem}
\newtheorem{theorem}{Theorem}[section]

\newtheorem{lemma}[theorem]{Lemma}

\newtheorem{corollary}[theorem]{Corollary}

\newtheorem{proposition}[theorem]{Proposition}

\newtheorem{prop-def}[theorem]{Proposition-Definition}
\newtheorem{lemma-def}[theorem]{Lemma-Definition}
\theoremstyle{definition}
\newtheorem{definition}[theorem]{Definition}
\newtheorem{example}[theorem]{Example}
\newtheorem{remark}[theorem]{Remark}
\newtheorem{notation}[theorem]{Notation}

\numberwithin{equation}{subsection}

\newcommand{\pp}{{\mathbb P}}

\newcommand{\ff}{{\mathbb F}}

\newcommand{\Cc}{{\mathscr C}}

\newcommand{\Dd}{{\mathscr D}}

\newcommand{\Aa}{{\mathscr A}}

\newcommand{\Gg}{{\mathscr G}}

\newcommand{\Hom}{{\rm Hom}}

\def\Hom{\mathrm{Hom}}

\def\Cat{{\mathscr C}{a}{t}}

\begin{document}

\author[Scott Balchin]{Scott Balchin}
\address{Department of Mathematics\\
University of Leicester\\
University Road, Leicester LE1 7RH, England, United Kingdom}
\email{slb76@le.ac.uk}

\title[Factorisation and Cohomology of Higher Categories]{Factorisation and Cohomology of Higher Categories}

\begin{abstract}{We construct an iterative method for factorising small strict $n$-categories into a unique (up to isomorphism) collection of small 1-categories. Following this we develop the theory to include a large class of $\infty$-categories. We use this factorisation to define cohomology theory for higher categories. Finally, we discuss some concrete examples of factorisations of low dimensional categories.}
\end{abstract}

\keywords{Factorisation, Strict $n$-Categories, Cohomology, Dendroidal Set}

\maketitle

\section*{Introduction}

Cohomology theories for small categories were developed by many authors, such as in the work of Watts \cite{watts}, Grothendiek \cite{groth}, Quillen \cite{quillen} and Baues-Wirsching \cite{baues}. Recently G\'alvez-Carrillo, Neumann and Tonks \cite{galvez} developed the theory of Thomason cohomology using Thomason natural systems. This cohomology uses general coefficient systems and was studied in the context of work by Gabriel-Zisman (see \cite{gabriel}).  There has been a recent surge in interest in higher category theory, more specifically that of $\infty$-categories. The goal of this article is to develop cohomology theories for these classes of categories.  Work undertaken by Lurie among others has given such theory for $\infty$-topoi and wider classes of $\infty$-categories, however the latter relies on taking the homotopy category of the $\infty$-category which disregards the higher morphisms \cite{luriebook} \cite{luriebook2}. These cohomology theories use local coefficients, however if we were to utilise the theory of Thomason cohomology we could have cohomology with very general coefficient systems.

In this article we introduce a way to factorise small strict higher categories into a tree like structure of 1-categories, which allows us to use classical methods to construct an $n$-group cohomology for an $n$-category.  Such factorisations have nice functorial properties and uniquely represent the category.  Given such a factorisation, we will look at properties of the shape of the factorisation, where the category of all shapes can be realised as a subcategory of the category of finite rooted trees $\Omega$ which plays a key role in the theory of dendroidal sets, see \cite{dendroidal2}.

The factorisation is given through the use of a shift map which utilises the morphisms in the category, which essentially is a decategorification from the objects up as opposed to the morphisms down as we would if we were taking the homotopy category.  Each shift reduces the categorical dimension by one, we then have reasoning to show that we need not consider each shift, but the homotopy category of each shift, which gives us our collection of 1-categories.

Although this theory was developed to study cohomology with general coefficient systems, it is evident that it may have broader applications, as the factorisation allows us to pass higher categories through classical cohomological machinery intended for 1-categories.

\section{Factorisation of Strict $n$-Categories}
\subsection{Shift Map}
We begin by introducing the so called shift map which allows us to reduce the dimension of a strict $n$-category, let us first introduce the necessary terminology (see \cite{leinster}).

\begin{definition}
A \emph{strict $n$-category} $\Cc$ is a category such that the morphisms are strict $(n-1)$-categories.  We denote by $\Cat_n$ the $(n+1)$-category of strict $n$-categories.
\end{definition}

\begin{remark}
For the rest of the article we will assume that $\Cc$ is a small $n$-category where we index the objects using the natural numbers, $X_i \in \Cc$ for $i \in \mathbb{N}$.
\end{remark}

\begin{definition}
Given $\Cc \in \Cat_n$ we construct a $(n-1)$-category, $s(\Cc)$, from $\Cc$ which we call the \emph{shift} of $\Cc$:

$$s(\Cc)=\coprod_{X_i, X_j \in \mathscr{C}} \Hom(X_i,X_j)$$

\noindent For ease of keeping track of the shifts, instead of writing $s(\Cc)$ we will write $\Cc[1]$.
\end{definition}

\begin{remark}
Note that in the above definition we do not restrict $i$ and $j$ such that $i \leq j$. This is done because $\Hom(X_i,X_j)$ is not necessarily equal to $\Hom(X_j,X_i)$, therefore we need to consider both cases.
\end{remark}

\begin{remark}
If we let $H_1=\Hom(X_i,X_j), H_2=\Hom(X_n,X_m) \in \Cc[1]$, then $\Hom(H_1,H_2)=\emptyset$ unless $i=n$ and $j=m$.  It follows that we get disjoint objects in $\Cc[1]$.  Therefore our above definition has a bit of ambiguity, we have chosen to consider $\Cc[1]$ as just one $(n-1)$-category, however we could consider it as a collection of disjoint ones, which we will be doing when we look at factorisations later.  Although we will use both ways, it should be apparent which one we are using.
\end{remark}

\begin{remark}
We have chosen $\Cc[1]$ to represent the shift map, because if we wanted to consider $s(s(\Cc))$, we can simply refer to it as $(\Cc[1])[1] = \Cc[2]$.  In general the $k$-th shift of $\Cc$ will be denoted by $\Cc[k]$.
\end{remark}

\begin{lemma}
Let $\Cc$ be a small strict $n$-category, then $\Cc[k]$ is defined for all $0 \leq k < n$
\end{lemma}

\begin{proof}
By definition, the shift map reduces the dimension by one each time it is applied.  Therefore it suffices to work out at which shift we get elements which are 1-categories. Clearly $\Cc[n-1]$ would have categorical dimension $n-(n-1)=1$.  Therefore we can only apply the shift map $n-1$ times.
\end{proof}

\subsection{Factorisation Trees and Planes}

Now that we have defined the shift map, we want to be able to use it to construct a factorisation for $\Cc$. $\Cc[1]$ can be regarded as many disjoint $(n-1)$-categories, and this is going to be an essential part of the construction.  The general idea is to factor $\Cc$ into its morphisms which are of one categorical dimension less, we then apply the shift map to each of the $\Cc[1]$ categories to factor each one of them.  We continue this process $(n-1)$ times until we reach a collection of 1-categories.  This process has a clear tree like structure which we will be discussing in detail later.

\begin{definition}
Let $\Cc$ be a strict $n$-category.  We inductively define the \emph{factorisation tree of $\Cc$}, which we denote by $\ff(\Cc)$, by the following procedure:

\begin{itemize}
\item The initial vertex is $\Cc$.
\item The second level has a vertex for each disjoint part of $\Cc[1]$ all of which are connected to the previous vertex.
\item The third level has a vertex for each disjoint part of $\Cc[2]$ which are connected to the vertex in the second level which they are the shift of.
\item Given the $k$-th level which is $\Cc[k-1]$, we attach the $\Cc[k]$ to the above vertices based on which one they are the shift of, which gives us the $(k+1)$-th level.
\end{itemize}
\end{definition}

\begin{remark}
Although in the definition we have said "vertex", we do not disregard what the category in question is, we still want to keep track of the category at each vertex.
\end{remark}

It is clear that $\ff(\Cc)$ uniquely determines $\Cc$, because we still have the initial category as our top vertex, therefore we repeat information. We wish to reduce the amount of information in the factorisation, but for it to still uniquely determine $\Cc$.

The first idea is to consider just the bottom row, however there is no unique way to embed a collection of 1-categories $\Cc_i$ into a strict 2-category $\Dd$ such that $\Dd[1]=\coprod \Cc_i$. An example of this is shown in the appendix. However we can see that if we knew the the objects of $\Dd$ then such an embedding would be unique (up to isomorphism). Therefore what we do is take the homotopy category at all levels of the shift.

\begin{definition}
Let $\Cc$ be a strict $n$-category, then the \emph{factorisation plane of $\Cc$}, denoted by $\pp(\Cc)$ is obtained by taking the homotopy category of every vertex of $\ff(\Cc)$.
\end{definition}

\begin{proposition}
Let $\Cc$ be a strict $n$-category, then $\pp(\Cc)$ uniquely determines $\Cc$ up to isomorphism.
\end{proposition}

\begin{remark}
When considering the factorisation, we took every shift, however we could consider a coarser factorisation by considering every $k$-th shift, although this would not uniquely factorise the category.
\end{remark}

\subsection{Functorial Properties of the Shift Map}
 We want to study the functorial properties of the shift map, and subsequently the factorisation of higher categories \cite{maclane}.

 \begin{proposition}
 Let $F: \Cc \to \Dd$ be a functor between strict $n$-categories, then there is an induced functor $F^\ast : \ff(\Cc) \to \ff(\Dd)$ which acts in a natural way on each vertex
 \end{proposition}

 \begin{proof}
 Functors are defined in such a way that they act nicely on the morphisms of the category, that is given an arrow \xymatrix{A \ar[r]^f & B} in $\mathscr{C}$, there is an arrow \xymatrix{FA \ar[r]^{Ff} & FB}. It follows that we have a naturally induced functor $F^\ast : \Hom(A,B) \to \Hom(FA,FB)$.  We can apply $F^\ast$ to each vertex of the factorisation to get the functor $F^\ast : \ff(\Cc) \to \ff(\Dd)$ as required.
 \end{proof}

 \begin{corollary}
Let $F: \Cc \to \Dd$ be a functor between strict $n$-categories, then there is an induced functor $F^\star : \pp(\Cc) \to \pp(\Dd)$ which acts in a natural way on each vertex
 \end{corollary}

\begin{proof}
This follows from the above theorem, $F^\star$ is $F^\ast$ acting only on the objects and 1-morphisms.
\end{proof}

 \begin{corollary}
 Let $\Cc \simeq \Dd$ be an equivalence of strict $n$-categories, then there is an equivalence $\ff(\Cc) \simeq \ff(\Dd)$
 \end{corollary}

\begin{proof}
If $F:\Cc \to \Dd$ is the functor giving the equivalence, then $F^\ast$ will be an equivalence on each part of the factorisation.
\end{proof}

\subsection{Example: Strict 2 and 3-Groups}

\begin{definition}
A \emph{strict $n$-group} is a group object internal to $(n-1)$-$\Gg rpd$, the $(n-1)$-dimensional category of strict groupoids \cite{garzon}.
\end{definition}

We begin by looking at the factorisations of $2$-groups, this is a relatively simple procedure as we only have one object in $\mathscr{C}$.  Let $\mathscr{G}$ be a strict $2$-group, then $\mathscr{G}[1]= \coprod \Hom(X_i,X_j) = \Hom(X,X)$.  That is the shift of $\mathscr{G}$ only has one category in it, which is a groupoid. If we represent by $\bullet$ a category, then $\mathbb{F}(\mathscr{G})$ has the following shape

\centerline{\xymatrix{ \mathscr{G} & \bullet \ar[d] \\
 \mathscr{G}[1] & \bullet}}

This case is trivial, the more interesting shapes arise when we deal with strict higher groups.  If we consider a strict $3$-group $\mathscr{G}$, then its first shift is the same as in the $2$-group case.  However the category in $\mathscr{G}[1]$ will have more than one object (unless it is the identity 3-group), and therefore $\mathscr{G}[2]$ will have more than one category in it. The shape for it is as in the following diagram:

\centerline{\xymatrix{ \mathscr{G} &&& \bullet \ar[d] \\
 \mathscr{G}[1] &&& \bullet \ar[d]  \ar[dll] \ar[dl] \ar[dr] \ar[drr]\\
 \mathscr{G}[2] & \bullet & \bullet & \cdots & \bullet & \bullet}}

 \begin{remark}
 In the above example we said that if $\mathscr{G}$ was the identity $3$-group, or the strict identity $3$-category, then we will only get one category in $\mathscr{G}[2]$.  This is easily realised for all identity categories.
\end{remark}

 \begin{definition}
 The \emph{strict identity $n$-category} is the category with only identity morphisms in each dimension, we will denote it by $\mathscr{I}_n$
 \end{definition}

 Let $\mathscr{I}_n$ be the identity $n$-category, then $\mathbb{F}(\mathscr{I}_n)$ has the following linear shape

 \centerline{\xymatrix{ \mathscr{I} & \bullet \ar[d] \\
 \mathscr{I}[1] & \bullet \ar[d]\\
 \mathscr{I}[2] & \bullet \ar[d]\\
 &  \vdots \ar[d] \\
 \mathscr{I}[n-1] & \bullet}}

\subsection{Factorisation of $\infty$-Categories}
We can easily extend the previous theory to the notion of $\infty$-categories if we choose the right model for our categories.

\begin{definition}
The category $\Cc at_\infty$ is the $\infty$-category of strict infinity categories.
\end{definition}

\begin{definition}
A \emph{strict $\infty$-category} $\Cc$ is a category such that for any two objects $\Hom(X,Y) \in \Cc at_\infty$.  That is we will consider an $\infty$-category to be a category enriched over the $\infty$-category of strict $\infty$-categories.
\end{definition}

\begin{definition}
The category $\Cc at_{\leq \infty}$ has for objects strict categories of all dimensions with morphisms between objects if and only if they have the same dimension.
\end{definition}

With this type of model considered, we can calculate the factorisation of a strict $\infty$-category, although the main difference will be that the factorisation tree will never terminate, it will be infinite.

\section{The Factorisation Tree in the Tree Category $\Omega$}
\subsection{The Category of Finite Rooted Trees}
We now look at the factorisation tree of a category in a more combinatorial manner. For this we consider their embedding into the category of finite rooted trees which we briefly review here (see \cite{dendroidal2}).

\begin{definition}
A \emph{finite symmetric rooted tree} is a finite poset $(T, \leq)$ such that it satisfies the following:

\begin{itemize}
\item There is a bottom element
\item For each element $e \in T$, the set $\{y \in T | y \leq e\}$ is a linear order under $\leq$
\item It is equipped with a subset $L \subset max(T)$ of the maximal elements of $T$.
\end{itemize}

The category of all all finite symmetric rooted trees is the tree category $\Omega$, where the morphisms are generated by face and degeneracy maps \cite{moerdijk}.
\end{definition}

\begin{notation}
Let $((T,\leq),L)$ be a tree, then:

\begin{itemize}
\item An element $e \in T$ is called an \emph{edge} of the tree
\item The bottom element is called the \emph{root} of the tree
\item An edge in $L \subset T$ is a \emph{leaf} of the tree
\item An edge which is either a leaf or the root is called an \emph{outer edge}
\item An edge which is neither a leaf nor the root is called an \emph{inner edge}
\item Given a non-leaf edge $e$, another edge $e'$ is called an \emph{incoming edge} if it is a direct successor of $e$.  We denote by $in(e)$ the set of incoming edges of $e$
\item Given a non-leaf edge $e$, the subset $v_e := \{e\} \cup in(e)$ is called a \emph{vertex} of $(T, \leq)$, We say that $e$ is the \emph{outgoing edge} of $v$, and we get $in(v_e) := in(e)$
\end{itemize}
\end{notation}

\begin{remark}
Given a finite rooted tree we can obtain a corresponding symmetric operad over $\mathscr{S}et$.  Therefore it follows that we have an inclusion $\Omega \hookrightarrow \mathscr{O}perad$ of the category of finite rooted trees to the category of symmetric operads (see \cite{weissop}).
\end{remark}

\subsection{Factorisation Trees to Finite Rooted Trees}
We now consider a method of embedding factorisation trees into $\Omega$.  We begin by considering degenerate branches of the factorisation tree

\begin{definition}
Let $\mathbb{F}(\mathscr{C})$ be the factorisation tree of a strict $n$-category $\mathscr{C}$.  If $\Hom(X_i,X_j) \in \mathscr{C}[k]$ is the identity $(n-k)$-category, then we say that this element and its successive shifts are a \emph{degenerate branch}.
\end{definition}

\begin{example}
The first example is that of the trivial $n$-category. We saw in an earlier example that the shape of the factorisation of the trivial $n$-category was a single line with $n$ vertices, in this case, the whole branch is degenerate.
\end{example}

\begin{definition}
Denote by $\mathbb{F}act$ the category whose objects are factorisations of categories, and whose morphisms are generated by deleting or inserting categories at vertices.
\end{definition}

\begin{remark}
For an overview of how face and degeneracy maps work in $\Omega$ see \cite{dendroidal2}.
\end{remark}

\begin{remark}
We can consider a factorisation to be a functor $\Cc at_{\leq \infty} \to \mathbb{F}act$.
\end{remark}

\begin{definition}
Let $\Cc$ be a strict $n$-category with factorisation $\mathbb{F}(\Cc)$, we define a quotient map on degenerate branches, where the first vertex of the degeneracy remains, giving it an leaf edge.  We denote this quotient map $Q : \mathbb{F}(\Cc) \to \mathbb{F}(\Cc) / (id)$.
\end{definition}

We now wish to consider the factorisation trees as finite rooted trees. This requires a forgetful functor $u: \mathbb{F}act \to \Omega$, which forgets the category at each vertex.

\begin{definition}
Let $\Cc$ be a strict $n$-category with factorisation $\mathbb{F}(\Cc)$, then the \emph{shape} of $\Cc$ is the tree given by $\mathbb{S}(\Cc)=u(\mathbb{F}(\Cc)/id) \in \Omega$
\end{definition}

\begin{remark}
It should be noted that unlike the factorisation itself, the shape of a category is not unique, there will be infinitely many categories with the same shape, for example all 1-categories have the same shape which is just a single vertex.
\end{remark}

\begin{example}
The identity categories of all dimensions greater than 1 have the same shape. This is because the single branch is degenerate so we quotient out by it and we are left with this shape

\centerline{\xymatrix{ \bullet \ar@{-}[d]\\
&}}
\end{example}

Recall that the simplex category $\Delta$ is a full subcategory of $\Omega$ \cite{moerdijk}. We want to know what families of category have shapes in $\Delta$.  These will be categories that have only one morphism in each dimension apart from the $n$-morphisms which must be non trivial. This gives us a countably infinite set of categories which have shape $[n] \in \Delta$.

\begin{definition}
We denote by $\Cc_{\Delta[n]}$ the collection of $n$-categories which have shape $[n] \in \Delta$, and call them the \emph{categories of $[n]$-shape}.
\end{definition}

\begin{example}
 The following image represents a category in $\Cc_{\Delta[3]}$ where the different line dashings represent different dimensions of morphisms
\vspace{5mm}

\centerline{
\includegraphics[scale=0.08]{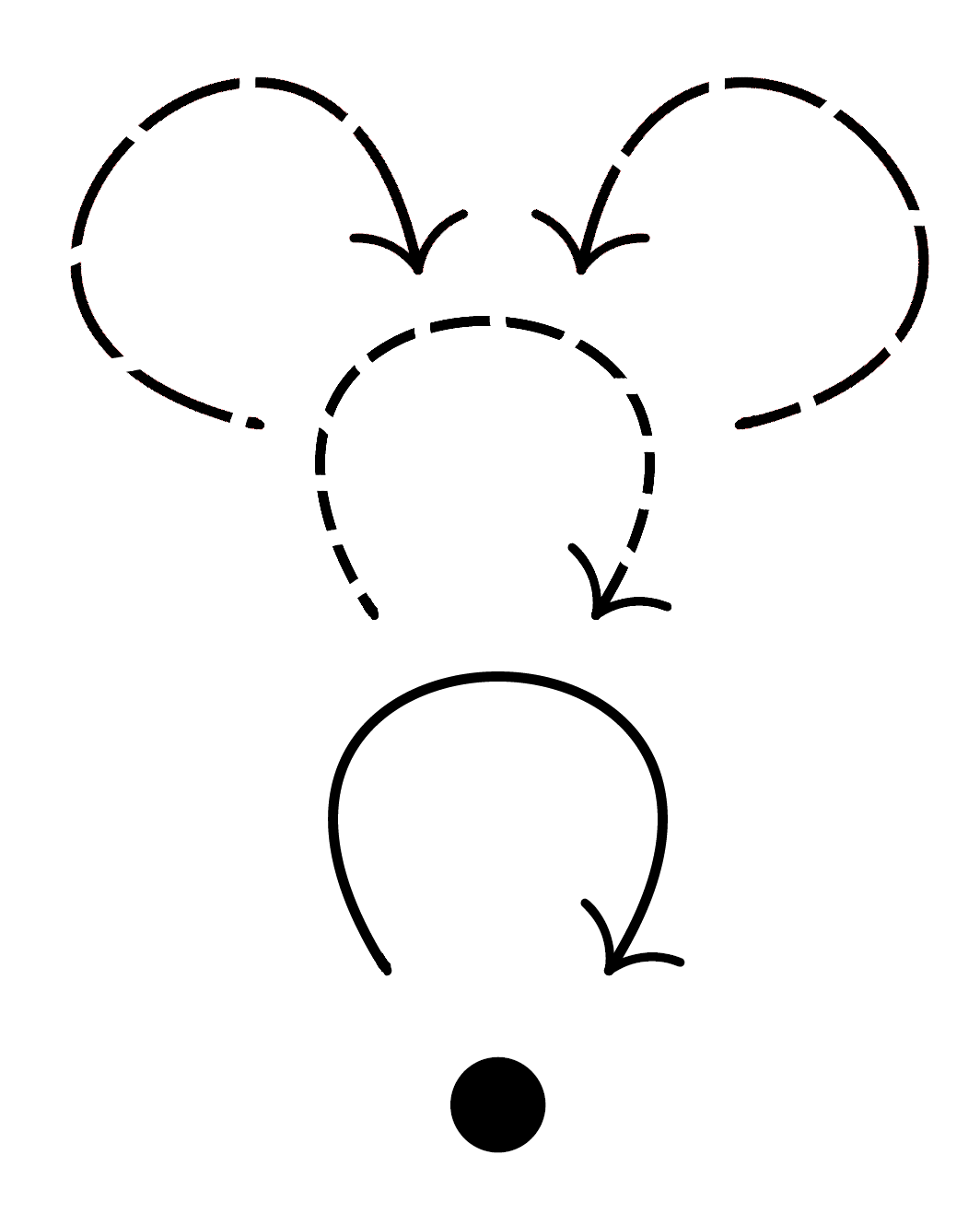}}

\vspace{5mm}
\end{example}

\begin{remark}
For the shape of a category to be truly a finite rooted tree we must specify the root and the direction. This follows quite naturally, the root would be an inward edge on the initial vertex (this is the one obtained from $\Cc$), and the direction would follow the only possible route.
\end{remark}


We now briefly discuss how operations in $\Omega$ affect the shape of a category.  If we consider the categories of simplex shape $\Cc_{\Delta[n]}$, then this is asking how do the face and degeneracy maps affect the category?  The face map $[n] \to [n-1]$ removes a vertex from the tree and merging the two edges which were incident to the vertex and this gives us the shape of $\Cc_{\Delta[n-1]}$ as expected. We have a similar process for the degeneracy maps which gives us a category of one dimension higher.

\section{Cohomology Theories for Strict $n$-Categories}
\subsection{Thomason Cohomology}
We now use Thomason cohomology as introduced in \cite{galvez}. This is a cohomology theory for small categories using general coefficient systems.  Using the factorisation system that we have developed, we apply the tools of Thomason cohomology to define cohomology for strict $n$-categories.  We begin by recalling some definitions from \cite{galvez}.

\begin{definition}
Let $\mathscr{C}$ be a small category, and $T:\Delta/ \mathscr{C} \to \mathscr{A}b$ be a Thomason natural system with values in the category of abelian groups.  The \emph{Thomason cochain complex} $C^\ast_{Th}(\mathscr{C},T)$ is defined as

\begin{center}
\[C^n_{Th}(\mathscr{C},T)= \prod_{([n],f)\in \Delta / \mathscr{C} } T(f) \]
\end{center}

For $n \geq 0$, we have the differential

\begin{center}
\[ d= \sum^{n+1}_{i=1} (-1)^i d_i : \prod_{[n] \to^f \mathscr{C}} T(f) \to \prod_{[n+1] \to^g \mathscr{C}} T(g) \]
\end{center}

induced by the coface maps.

The \emph{$n$-th Thomason cohomology} of $\Cc$ with coefficients $T$ is defined as $H^n_{Th}(\mathscr{C},T):=H^n(C^\ast_{Th}(\mathscr{C},T),d)$
\end{definition}

\begin{remark}
In the definition above we used $\mathscr{A}b$, the category of abelian groups, however the definition works with any abelian category $\Aa$ with some conditions (see \cite{galvez}).
\end{remark}

We now use this, and our factorisation plane to define the Thomason cohomology, but first we introduce some notation

\begin{definition}
Let $\Cc$ be a strict $n$-category with factorisation plane $\mathbb{P}(\Cc)$. We denote by $\mathscr{C}[i,j]$ the category that is placed $i$-th vertically from the top and $j$-th horizontally from the left of $\mathbb{P}(\Cc)$.
\end{definition}

\begin{example}
Consider a category with the following factorisation
\vspace{5mm}

\centerline{\xymatrix{ &&& \bullet \ar[dl] \ar[dr] \\
& & \bullet \ar[dl] \ar[dr] & & \bullet \ar[d] \\
& \bullet & & \star &  \bullet}}
\vspace{5mm}

Then the category denoted by $\star$ would be $\Cc[3,2]$

\end{example}

\begin{definition}
Let $\Cc$ be a strict $n$-category with factorisation plane $\mathbb{P}(\Cc)$.  Then a Thomason natural system on $\Cc$ is a collection of functors $T_{[i,j]}:\Delta/( \Cc[i,j] )\to \Aa b$
\end{definition}

\begin{definition}
Let $\Cc$ be a strict $m$-category with factorisation plane $\mathbb{P}(\Cc)$.  Then the Thomason cohomology group of $\Cc$ is the strict $m$-group whose $k$-morphisms are the collection are groups $H^n_{Th}(\mathscr{C}[k,j],T_{[k,j]})$ where $j$ ranges over the obvious domain.  We attach the $k$ morphisms to the $k-1$ morphisms by respecting the shape of the tree. That is, we attach $H^n_{Th}(\mathscr{C}[k,j],T_{[k,j]})$ as $k$-morphisms of $H^n_{Th}(\mathscr{C}[k-1,j'],T_{[k-1,j']})$ if there is an edge from $\mathscr{C}[k-1,j']$ to $\mathscr{C}[k,j]$ in $\mathbb{P}(\Cc)$.

We will denote the $m$-group $n$-th cohomology by $H^{m,n}_{Th}(\mathscr{C},T)$
\end{definition}

\begin{remark}
We can easily apply the above procedure to the class of $\infty$-categories that were discussed earlier, in this case we would get a strict $\infty$-group for the cohomology.
\end{remark}

We can use the functorial properties of the cohomology theory and the factorisations to give some results about the cohomology of higher categories.

\begin{proposition}

Let $\Cc$ and $\Dd$ be strict $n$-categories, and $(\phi,\psi)$ a pair of adjoint strict $n$-functors

\begin{center}
$ \phi : \Dd \rightleftarrows \Cc : \psi$
\end{center}

Then there is a natural isomorphism $H^{n,\ast}_{Th}(\Dd,\phi^\ast(T)) \cong H^{n,\ast}_{Th(n)}(\Cc,T)$
\end{proposition}

\begin{proof}
Using corollary $1.14$ we can lift $\phi$ and $\psi$ to $\phi^\star$ and $\psi^\star$ respectively. These are a pair of adjoint functors between $\mathbb{P}(\Cc)$ and $\mathbb{P}(\Dd)$ which act on each node in the factorisaion tree in the obvious way.  Then using results from \cite{galvez} concerning the functorial properties of Thomason cohomology, we can conclude our result.
\end{proof}

Using the methods outlined in \cite{galvez}, it is possible to give similar definitions for Baues-Wirsching cohomology and other classical cohomology theories for strict $n$-categories.

\section{Example: A $3$-Groupoid}
\subsection{Factorisation Tree and Plane}
In this section we will be considering a strict 3-groupoid $\Cc$ which is the following category

\centerline{
\includegraphics[scale=0.08]{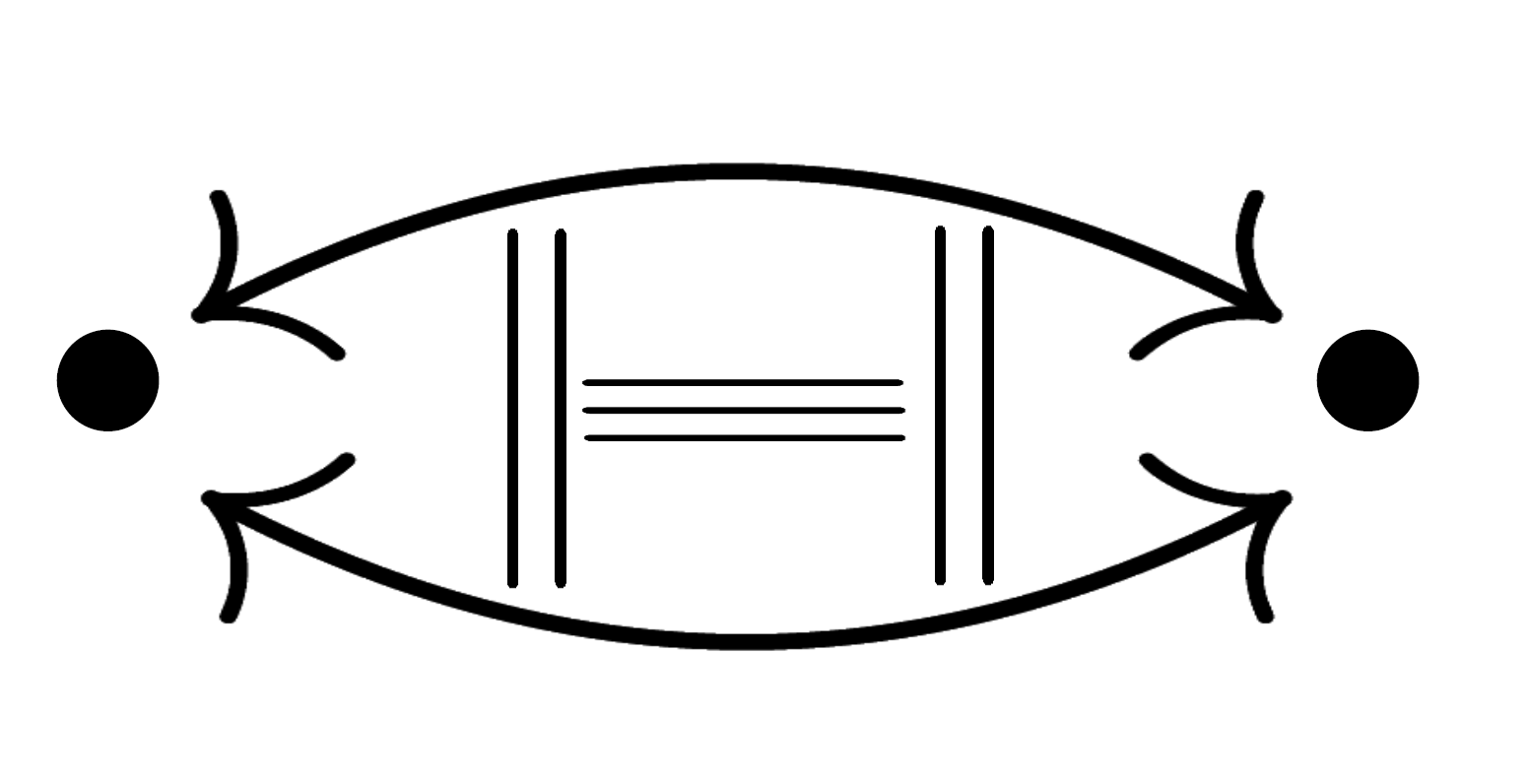}}

Here the double lines are 2-morphisms and the triple line is a 3-morphism.  The identity 1-, 2- and 3-morphisms are not included in the image.  The first step in calculating the factorisation of $\Cc$ is to calculate $\Cc[1]$.  As $\Cc$ is a groupoid which is fully connected with two objects, we have $2^2=4$ 2-categories in $\Cc[1]$. Let us call the left object $X_1$ and the right object $X_2$

Begin by noting that $\Hom(X_1,X_1)=\Hom(X_2,X_2)=\mathscr{I}_2$, where $\mathscr{I}_2$ is the identity 2-category.  Also, because of the symmetry of the category we will have that $\Hom(X_1,X_2)=\Hom(X_2,X_1)$, we just need to find out what 2-category this is. An easy way to do this is to regard the morphisms as the new objects, and then the 2-morphisms as the new 1-morphisms and so on. It follows that $\Hom(X_1,X_2)$ is the following category

\vspace{5mm}

\centerline{
\xymatrix{\mathscr{C}'= &\bullet\rtwocell{=} & \bullet \ltwocell{=} }}
\vspace{5mm}

These categories make up $\Cc[1]$, we now need to work out the shift of each of these categories so that we can calculate $\Cc[2]$. Clearly $\mathscr{I}_2[1]=\mathscr{I}_1$, so we need only consider $\Cc'[1]$ in any detail.  As with $\Cc$ we have two objects which are connected, so we will get four 1-categories in $\Cc'[1]$.  If we label the left object of $\Cc'$ $Y_1$ and the right object $Y_2$, then we have that $\Hom(Y_1,Y_1)=\Hom(Y_2,Y_2)=\mathscr{I}_1$, and $\Hom(Y_1,Y_2)=\Hom(Y_2,Y_1)=\Cc''$ where $\Cc''$ is given by the following category

\vspace{5mm}

\centerline{
\xymatrix{\mathscr{C}''= &\bullet \ar[r] & \bullet \ar[l] }}
\vspace{5mm}

We now have all of the parts required to derive the factorisation of $\Cc$. $\mathbb{F}(\Cc)$ is given by the following tree:

\vspace{5mm}

\centerline{
\xymatrix{\Cc &&&&&& \Cc \ar[ddlllll] \ar[ddrrrrr] \ar[ddll] \ar[ddrr]\\
\\
\Cc[1] & \mathscr{I}_2 \ar[dd] &&& \Cc' \ar[ddll] \ar[ddl] \ar[dd] \ar[ddr]  &&&& \Cc' \ar[ddl] \ar[dd] \ar[ddr] \ar[ddrr] &&& \mathscr{I}_2 \ar[dd]\\
\\
\Cc[2] & \mathscr{I}_1 & \mathscr{I}_1 & \Cc'' & \Cc'' & \mathscr{I}_1 && \mathscr{I}_1 & \Cc'' & \Cc'' & \mathscr{I}_1 & \mathscr{I}_1}}
\vspace{5mm}

The next task is to find the factorisation plane of $\Cc$, so this requires us to find the homotopy categories of the involved higher categories.  Note that $h(\mathscr{I}_2)=\mathscr{I}_1$ and that $h(\Cc)=h(\Cc')=\Cc''$, so that $\mathbb{P}(\Cc)$ is the following tree

\vspace{5mm}

\centerline{
\xymatrix{h(\Cc) &&&&&& \Cc'' \ar[ddlllll] \ar[ddrrrrr] \ar[ddll] \ar[ddrr]\\
\\
h(\Cc[1]) & \mathscr{I}_1 \ar[dd] &&& \Cc'' \ar[ddll] \ar[ddl] \ar[dd] \ar[ddr]  &&&& \Cc'' \ar[ddl] \ar[dd] \ar[ddr] \ar[ddrr] &&& \mathscr{I}_1 \ar[dd]\\
\\
\Cc[2] & \mathscr{I}_1 & \mathscr{I}_1 & \Cc'' & \Cc'' & \mathscr{I}_1 && \mathscr{I}_1 & \Cc'' & \Cc'' & \mathscr{I}_1 & \mathscr{I}_1}}
\vspace{5mm}

Finally we would like to know the shape of $\Cc$. This requires us to quotient out the degenerate branches. We have two branches that are degenerate, both of which start at $\mathscr{I}_2$. When we quotient out these and use the forgetful functor (from definition 2.9) we get the following rooted tree in $\Omega$

\vspace{5mm}

\centerline{
\xymatrix{ &&&&& \ar[d]\\
&&&&& \bullet \ar[ddlllll] \ar[ddrrrrr] \ar[ddll] \ar[ddrr]\\
\\
  \bullet \ar[dd] &&& \bullet \ar[ddll] \ar[ddl] \ar[dd] \ar[ddr]  &&&& \bullet \ar[ddl] \ar[dd] \ar[ddr] \ar[ddrr] &&& \bullet \ar[dd]\\
\\
 & \bullet & \bullet & \bullet & \bullet && \bullet & \bullet & \bullet & \bullet & }}
\vspace{5mm}

We now have derived $\mathbb{F}(\Cc)$, $\mathbb{P}(\Cc)$ and $\mathbb{S}(\Cc)$.  Now we consider $\mathbb{P}(\Cc)$ so we can calculate the cohomology of $\Cc$. We will not explicitly calculate each homology group, but just leave them as $H^n_{Th}(\mathscr{C}[k,j],T_{[k,j]})$, where $k=1,2,3$ and $j$ can take a range of values respective of the value of $k$. Then we have $H^{3,n}(\Cc,T)$ using the following form, if a group has an arrow to another group, then it indicates that that group acts on the group it is incident to, and the $\bullet$ represents an object.  We will also shorten $H^n_{Th}(\mathscr{C}[k,j],T_{[k,j]})$ to $H^n(k,j)$ for space purposes

\vspace{5mm}


\centerline{\resizebox{35em}{0.8em}{
\xymatrix{ H^n(3,1) \ar[d] & H^n(3,2) \ar[drr] & H^n(3,3) \ar[dr] & H^n(3,4) \ar[d] & H^n(3,5) \ar[dl] && H^n(3,6) \ar[dr] & H^n(3,7) \ar[d] & H^n(3,8) \ar[dl] &  H^n(3,9) \ar[dll] & H^n(3,10)\ar[d] \\
H^n(2,1) \ar[drrrrr] &&& H^n(2,2) \ar[drr] &&&& H^n(2,3)\ar[dll] &&& H^n(2,4) \ar[dlllll]\\
&&&&& H^n(1,1) \ar[d] \\
&&&&& \bullet}}}

\vspace{5mm}

\appendix
\section{Unique Embedding of 1-Categories into Strict 2-Categories}

We will first show that there is always a trivial embedding of a 1-category into a strict 2-category, such that the shift of the 2-category is our collection of 1-categories.

Suppose that we start with the 2-groupoid $\mathscr{C}$ displayed below discarding the identity 2 morphisms.
\vspace{5mm}

\centerline{
\xymatrix{\mathscr{C}= &\bullet \ar@(ul,dl) \ar@(dl,ul) \rtwocell{=} & \bullet \ltwocell{=} \ar@(ur,dr) \ar@(dr,ur)}}
\vspace{5mm}

We want to derive $\mathscr{C}[1]$, as the category is connected, and we have two objects, we will get four 1-categories from the shift. This gives us the following shift $\Cc[1]$:
\vspace{5mm}

\centerline{\xymatrix{\mathscr{C}[1]= & \bullet \ar@(ul,dl) \ar@(dl,ul) & \bullet \ar@(ul,dl) \ar@(dl,ul) \ar[r] & \bullet \ar[l] \ar@(ur,dr) \ar@(dr,ur) & & \bullet \ar@(ul,dl) \ar@(dl,ul) \ar[r] & \bullet \ar[l] \ar@(ur,dr) \ar@(dr,ur) &  \bullet \ar@(ur,dr) \ar@(dr,ur)}}

\vspace{5mm}

Now we want to show that this factorisation is not unique in the sense that another 2-category could have this collection of 1-categories as its shift.  The procedure is easy, we can just take four disjoint objects, and embed each shift category above as a set of morphisms going from the object to itself.  Then if we take the shift of this 2-category, because the objects are disjoint, we get four 1-categories and they are exactly as we require.

\noindent We will now look at another example, where there is a non-trivial embedding of a collection of 1-categories into two different 2-categories.  The 2-category in question has three objects and looks like this, where we have excluded the identity 1 morphisms as well as the identity 2-morphisms to make the diagram less cluttered.
\vspace{5mm}

\centerline{
\xymatrix{\mathscr{C}= & \bullet  \rrtwocell{=} \ddrtwocell{=} && \bullet \lltwocell{=} \ddltwocell{=}\\
& \\
&& \bullet \uurtwocell{=} \uultwocell{=}  }}

\vspace{5mm}

The shift of $\mathscr{C}$ has two different types of structure, with multiplicity 3 and 6 which we will denote by $"3 \times"$ and $"6 \times"$.  Note that we have $9 = 3^2$ categories as all objects are linked by morphisms in both directions.

\vspace{5mm}

\centerline{
\xymatrix{\mathscr{C}[1] = & 3 \times & \bullet \ar@(ul,dl) \ar@(dl,ul) & 6 \times & \bullet \ar@(ul,dl) \ar@(dl,ul) \ar[r] & \bullet \ar[l] \ar@(ur,dr) \ar@(dr,ur)}}

\vspace{5mm}

We want to find another category $\mathscr{D}$ which has $\mathscr{C}[1]=\mathscr{D}[1]$.  As we have nine segments above, it would make sense to look at a one category with five objects, which is split into three parts, with one object on its own and two sets of two, this gives us $1^2+2^2+2^2=9$ segments in the shift.  Such a category is given below, and it can be easily verified that this satisfies $\mathscr{C}[1]=\mathscr{D}[1]$, again, excluding all identity morphisms.

\vspace{5mm}

\centerline{
\includegraphics[scale=0.3]{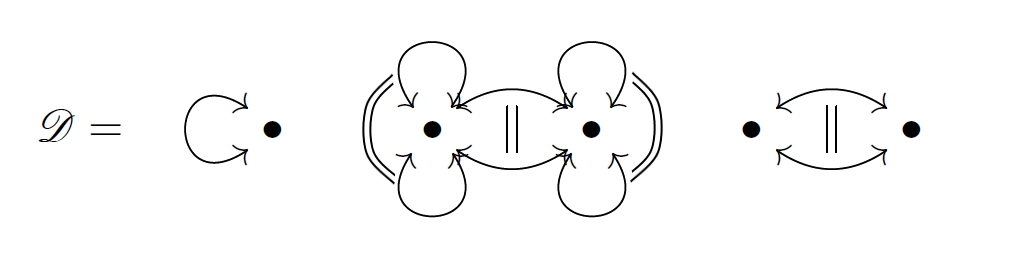}}


\noindent This is a non trivial embedding, so given a collection of 1-categories, we do not have a unique way to embed them into a 2-category, and therefore the entire factorisation tree cannot be determined from a subset of levels from the whole tree.  Given the homotopy category of the above level we know the number of objects in the category. This tells us how many objects we are embedding our category into, and how connected they are.  Given this information, we get a unique embedding, which leads to the reasoning that $\mathbb{P}(\Cc)$ represents the same information as $\mathbb{F}(\Cc)$.


\vspace*{0.3cm} {\it Acknowledgements.}
I would like to thank Frank Neumann for his encouragement and suggestions and also for the initial motivation to study higher category theory.

\bibliographystyle{amsplain}
\bibliography{shiftbib}
\end{document}